\documentclass[12pt,a4paper]{amsart}

\usepackage[utf8]{inputenc}
\usepackage[english]{babel}
\usepackage{amsmath,amsfonts,amssymb}
\usepackage{amsthm}
\usepackage[hmarginratio=1:1]{geometry}
\usepackage{verbatim}

\usepackage[pdfencoding=auto]{hyperref}
\usepackage[nameinlink]{cleveref} 

\usepackage{xcolor}
\hypersetup{
    colorlinks,
    linkcolor={blue!50!blue},
    citecolor={blue!50!blue},
    urlcolor={blue!80!black}
}
\usepackage{enumitem}

\theoremstyle{plain}
  \newtheorem{main-theorem}{\bf Theorem}
  \newtheorem{theorem}{\bf Theorem}[section]
  \newtheorem{proposition}[theorem]{\bf Proposition}
  \newtheorem{lemma}[theorem]{\bf Lemma}

\theoremstyle{remark}

\DeclareMathOperator*{\supp}{supp} 
  
\newcommand{\RR}{\mathbb{R}} 
\newcommand{\NN}{\mathbb{N}} 
\newcommand{\EE}{\mathbb{E}}  
\newcommand{\PP}{\mathbb{P}} 
\newcommand{\WW}{\mathbb{W}} 
\newcommand{\dd}{\mathrm{d}}

\newcommand{\norm}[1]{\lVert#1\rVert} 

\newcommand{\br}[1]{\langle #1 \rangle } 

\title[Observability with noise]{The observational limit of wave packets with noisy measurements}

\author[P. Caro]{Pedro  Caro}
\address{Pedro Caro, BCAM- Basque Center of Applied Mathematics}
\email{pcaro@bcamath.org}

\author[C. J. Meroño]{Cristóbal J. Meroño}
\address{Cristóbal J. Meroño, Universidad Politécnica de Madrid, ETSI Caminos, Departmento de Matemática e Informática}
\email{cj.merono@upm.es}

\date{\today}
\thanks{PC is supported by BERC 2018-2021, Severo Ochoa SEV-2017-0718, Ikerbasque and PGC2018-094528-B-I00. CJM was supported by Spanish Grant MTM2017-85934-C3-3-P and wants to thank BCAM for its support and hospitality.}
\begin{document}

\begin{abstract}
The authors consider the problem of recovering an observable from 
certain measurements containing random errors. The observable is 
given by a pseudodifferential operator while the random errors are 
generated by a Gaussian white noise. The authors show how wave 
packets can be used to partially recover the observable from the 
measurements almost surely. Furthermore, they point out the 
limitation of wave packets to recover the remaining part of the 
observable, and show how the errors hide the signal coming from the 
observable. The recovery results  are 
based on an ergodicity property of the errors produced by wave 
packets.
\end{abstract}

\maketitle

\section{Introduction and main results} \label{sec_intro}

We study the problem of recovering an \textit{observable} $P$ from certain 
\textit{measurements} $\mathcal{N}_P$ that contain some random 
errors. 
In our case, the observable $P$ is 
described by a pseudodifferntial operator
\begin{equation} \label{eq:def_P}
P f (x) = \frac{1}{(2\pi)^{d/2}} \int_{\RR^d} e^{i x \cdot \xi} a(x,\xi) \widehat{f}(\xi) \, \dd \xi,
\end{equation}
with a classical symbol $a$ of order $m \in \RR$. Here $d \geq 1$ and $\widehat{f}$ 
denotes the Fourier transform of $f$ defined as
$\widehat{f}(\xi) = 1/(2\pi)^{d/2}\int_{\RR^d} e^{-i\xi\cdot x} f(x) \,\dd x$.
The measurements $\mathcal{N}_P$ consists of the ideal data or \textit{signal} yielded by $P$ 
plus a centered complex Gaussian variable modeling the random error. More precisely, 
for arbitrary \textit{states} $f$ and $g$, the measurements are given by
\[\mathcal{N}_P (f, g) = \int_{\RR^d} \overline{f} P g \, + \mathcal{E}(\overline{f}, g),\]
where $\mathcal{E}(\overline{f}, g)$ is a complex Gaussian with zero mean, and 
variance depending on the states $f$ and $g$. For the purpose of this 
introduction, we think of the random error given as follows ---a more 
complete description can be found in the \cref{sec_noise}. Let $(\Omega, \mathcal{F}, \PP)$ be a probability space so that there 
exists a countable family $\{ X_\alpha : \alpha \in \NN^2 \}$ of independent 
complex Gaussian variables satisfying
\begin{equation} \label{eq:complexGaussian}
\EE X_\alpha = 0,  \qquad \EE (\overline{X}_\alpha X_\alpha) = 1.
\end{equation}
Then, the \textit{error} $\mathcal{E} = \mathcal{E}_\beta $, contained in the measurements $\mathcal{N}_P = \mathcal{N}_{\beta, P} $, can be expressed as
\begin{equation}
\mathcal{E}_\beta (f, g) = \sum_{\alpha \in \NN^2} (f|e_{\alpha_1})_\beta (g|e_{\alpha_2})_\beta X_\alpha,
\label{id:error}
\end{equation}
where $\alpha = (\alpha_1, \alpha_2)$, $\{ e_n : n \in \NN \}$ is an 
orthonormal basis of the Sobolev space $H^\beta (\RR^d)$ with $\beta \in \RR$ and 
$(\centerdot|\centerdot)_\beta$ denotes its inner product,
\begin{equation} \label{eq:inner_p}
(f|g)_\beta = \int_{\RR^d} (1 + |\xi|^2)^\beta \overline{\widehat{f}(\xi)} \widehat{g}(\xi) \, \dd \xi.
\end{equation}
As usual, $\norm{f}_{\beta}^2 = (f|f)_\beta$. For convenience, if $\beta = 0$, we simplify the notation so that the inner product is denoted by $(\centerdot|\centerdot)$ and the norm by $\norm{\centerdot}$.
In the \cref{sec_noise} we shall show that
\begin{equation} \label{eq:varianza}
\EE \, \mathcal E_\beta(f,g) = 0 \quad \textnormal{and} \quad  \EE \left | \mathcal E_\beta(f,g) \right |^2 = \norm{f}_\beta^2 \norm{g}_\beta^2.
\end{equation}
Note that different values of $\beta$ might correspond to different variances of the error $\mathcal{E}_\beta$. This can be interpreted as follows. If $\beta = 0$ the oscillations of the states $f$ and 
$g$ do not affect the size of the error, only their masses influence 
its size; while for $\beta>0$ the noise can be increased in the 
cases where the states $f$ and $g$ present many oscillations. When 
$\beta<0$ the oscillations of the states might reduce the size of 
the noise.

For simplicity, we only consider observables $P$ so that their symbols accept 
an expansion in terms of homogeneous functions. More precisely, we consider 
an observable $P$ so that its symbol $a$ of order $m$ satisfies that
\begin{equation} \label{eq:ps_1}
a \sim \sum_{j=1}^\infty a_j,
\end{equation}
with $a_j$ being a classical symbol of order $m_j \in \RR$, for
$m_j < m_{j-1} < \dots < m_1 = m$, which is \textit{homogeneous in the variable $\xi$}, that is,  
\begin{equation} \label{eq:ps_2}
a_j(x,t\xi) = t^{m_j} a_j(x,\xi),  \quad \textnormal{whenever } |\xi|\ge 1/2  \textnormal{ and } t\ge 1.
\end{equation}
Note that such an expansion  is not unique ---one can choose any smooth extension of the $a_j(x,\xi)$ for $|\xi|\le 1/2$.  In the \cref{sec_deterministic}, we provide more details 
about classical symbols and the exact meaning of the expansion 
\eqref{eq:ps_1}. It can be convenient for the reader to note that differential operators of 
any order with smooth coefficients are observable satisfying these 
properties.

Note that recovering the symbol $a$ 
of a pseudodifferential operator $P$ is an easy task: if one extends 
the operator $P$ to the tempered distributions in $\RR^d$, one can 
verify that \textit{plane waves} $e^{ix\cdot \xi}$ yield the symbol, $  e^{-ix\cdot\xi} P(e^{ix\cdot\xi}) = a(x,\xi)$. In the case of a 
differential 
operator
$P = \sum_{|\alpha|\leq m} a_\alpha(x) (-i)^{|\alpha|} \partial_x^\alpha, $
this is straightforward $e^{-ix\cdot \xi} P (e^{ix\cdot \xi}) = \sum_{|\alpha|\leq m} a_\alpha(x) \xi^\alpha $.
This suggests 
that, to recover the observable in the presence of a random noise, one needs 
to approximate plane waves with functions in $H^\beta(\RR^d)$. The problem is 
that this cannot be achieved with sequences that  are uniformly bounded in 
$H^\beta(\RR^d)$ since no Sobolev space in $\RR^d$ contains the plane 
waves. This seems a relevant point since, if the sequence is not bounded, by 
\eqref{eq:varianza}, it looks unlikely that we can filter the noise. 
Fortunately, there is a way around this at least when the symbol has an 
homogeneous expansion as in \eqref{eq:ps_1} and \eqref{eq:ps_2}. Indeed, it is enough to 
introduce a family of states $\{ f_t : t \geq 1 \}$
that concentrate around a point in $\RR^d$ as $t$ grows, and oscillate at a higher order 
than the rate of concentration, namely, wave packets.

At this point, we are ready to state the main results of this 
article. The first theorem consists of several
reconstruction formulas that will be used recursively to obtain the 
expansion of the observable $P$ up to some extent. Let us first state 
this first theorem and then we explain how to use these formulas to obtain the 
expansion.

\begin{main-theorem} \label{thm_main_1} \sl Let $P$ be an observable whose symbol $a$ of order $m$ accepts an homogeneous expansion as in \eqref{eq:ps_1}, and $Q_j$ denote the pseudodifferential operators with symbols $a_j$. Let $\beta \in \RR$ be the value in the error $\mathcal{E}_\beta$ and assume that $m > 2\beta - 1/2$. Set
\[k_\beta = \min \{ j \in \NN : m_{j+1} \leq 2\beta - 1/2 \} \textnormal{ and } j_\beta = \min \{ j \in \NN \cup \{0 \} : m_{j+1} \leq 2\beta \},\]
and note that $k_\beta \geq 1$ while $0 \leq j_\beta \leq k_\beta$.
For every $x_0, \xi_0 \in \RR^d$ with $|\xi_0| = 1$, $\lambda > 1$ and $t \geq 1$, set the \textit{wave packets}
\[f_{t,\lambda}(x) =   t^{d/2}\chi(t(x-x_0))e^{it^{\lambda} (x-x_0)\cdot \xi_0},\]
where $\chi$ is a function in the Schwartz class satisfying \eqref{eq:chi_properties} in the \cref{sec_deterministic}.
Assuming that we know $\beta$ and the orders $\{ m_1 \dots m_{k_\beta} \}$, we have:
\begin{enumerate}[label=(\alph*)]
\item \label{lim:non-averaging} Whenever $j_\beta \geq 1 $, set $\lambda_j = \max (1/(m_j - m_{j+1}), 2)$ if $1 \leq j < j_\beta$, and $\lambda_{j_\beta} \geq \max (1/(m_{j_\beta} - 2\beta), 2)$. Then,
\[\lim_{N \to \infty} N^{-\lambda_j m_j} \big[\mathcal{N}_{\beta, P} (f_{N, \lambda_j}, f_{N, \lambda_j}) - \sum_{0<k<j} (f_{N, \lambda_j}| Q_k f_{N, \lambda_j}) \big] = a_j (x_0, \xi_0) \]
almost surely for every $j \in \{ 1, \dots, j_\beta \}$.
\item \label{lim:averaging} Set $\lambda_j = \max (1/(m_j - m_{j+1}), 2)$ if $j_\beta < j < k_\beta$, and $\lambda_{k_\beta} > \max (1/(m_{k_\beta} - 2\beta + 1/2), 2)$. Then,
\[\lim_{N \to \infty} \frac{1}{N} \int_N^{2N} t^{-\lambda_j m_j} \big[\mathcal{N}_{\beta, P} (f_{t, \lambda_j}, f_{t, \lambda_j}) - \sum_{0<k<j} (f_{t, \lambda_j}| Q_k f_{t, \lambda_j}) \big] \, \dd t = a_j (x_0, \xi_0) \]
almost surely whenever $0 \leq j_\beta < j \leq k_\beta$.
\end{enumerate}
\end{main-theorem}

To explain how to use this formulas to obtain the expansion, assume for instance that the observable $P$ is so that $m > m_{j_\beta} > 2 \beta \geq m_{k_\beta} > 2 \beta - 1/2$. Choosing $\lambda_1$ as described in \ref{lim:non-averaging},
we compute $a_1$ from
\[\lim_{N \to \infty} N^{-\lambda_1 m_1} \big[\mathcal{N}_{\beta, P} (f_{N, \lambda_1}, f_{N, \lambda_1}) \big] = a_1 (x_0, \xi_0), \]
which is a limit that converges almost surely. Combining this with the homogeneity property \eqref{eq:ps_2} one essentially obtains $a_1 (x, \xi) $  for all $x \in \RR$ and $|\xi|\ge 1/2$, to define the values of $a_1 (x, \xi) $ when $|\xi| \le 1/2$, one can choose any smooth extension. Once we know $a_1$, we construct $Q_1$ and compute $a_2$ from
\[\lim_{N \to \infty} N^{-\lambda_2 m_2} \big[\mathcal{N}_{\beta, P} (f_{N, \lambda_2}, f_{N, \lambda_2}) -  (f_{N, \lambda_2}| Q_1 f_{N, \lambda_2}) \big] = a_2 (x_0, \xi_0), \]
whose convergence is almost surely. Iterating the process ---still applying \ref{lim:non-averaging}--- we obtain $a_j$ with $j \leq j_\beta$.  Note that the choices of $\lambda_1, \dots, \lambda_{j_\beta}$ are only based on the a priori knowledge of $\beta$ and $ \{m_1, \dots, m_{j_\beta} \}$. In order to compute $a_{j_\beta + 1}$, we use the limit in \ref{lim:averaging}:
\[\lim_{N \to \infty} \frac{1}{N} \int_N^{2N} t^{-\lambda_{j_\beta + 1} m_{j_\beta + 1}} \big[\mathcal{N}_{\beta, P} (f_{t, \lambda_{j_\beta + 1}}, f_{t, \lambda_{j_\beta + 1}}) - \sum_{k=1}^{j_\beta} (f_{t, \lambda_{j_\beta + 1}}| Q_k f_{t, \lambda_{j_\beta + 1}}) \big] \, \dd t, \]
which again converges almost surely to $a_{j_\beta + 1} (x_0, \xi_0)$. Iterating the process ---applying now \ref{lim:averaging}--- we obtain $a_j$ for $j_\beta < j \leq k_\beta$. Again, the choices of $\lambda_j$ with $j \in \{ {j_\beta + 1}, \dots, {k_\beta} \}$ are only based on the a priori knowledge of $\beta$ and the orders $ \{m_{j_\beta + 1}, \dots, m_{k_\beta} \}$. This completes the iteration to construct $a_1, \dots, a_{k_\beta}$ in the situation where the observable $P$ satisfies $m > m_{j_\beta} > 2 \beta \geq m_{k_\beta} > 2 \beta - 1/2$. Note that, in the case where $j_\beta = 0$, we would only use \ref{lim:averaging} to recover $a_1, \dots, a_{k_\beta}$ in the expansion.

It is worth to pay 
attention to the case where the observable $P$ is a differential 
operator of order $m$, in that case it is possible to recover the 
full $P$ from $\mathcal{N}_{\beta, P}$ with $\beta < 1/4$. This means 
that, even with an error $\mathcal{E}_\beta$ that gets amplified with 
the oscillations of the states,
one can obtain the full observable. Furthermore, in the 
case that $P$ is elliptic ---the leading order term of its symbol 
satisfies that $\sum_{|\alpha| = m} a_\alpha(x) \xi^\alpha \neq 0$ 
for all $\xi \in \RR^d \setminus \{ 0 \}$,  one can construct a 
suitable parametrix 
from the measurements $\mathcal{N}_{\beta, P}$ with $\beta < 1/4$.

\Cref{thm_main_1} shows that wave packets are appropriate states to 
obtain part of the expansion of the observable $P$. Furthermore, as 
we will show in the \cref{sec_deterministic}, they are suitable to 
recover the full expansion in the case of ideal data, that is, in the 
absence of $\mathcal{E}_\beta$. However, in the presence of noise, we 
cannot recover $a_j$ without uncertainty for any $j > k_\beta$ since the variance of the errors becomes so large that the signal produced by the observable is lost in the noise. This 
claim is actually the statement of the second theorem of this paper. 
For convenience, introduce the following notation
\begin{equation}
\mathcal{N}_{\beta, P_j} (f,g) = \mathcal{N}_{\beta, P} (f,g) - \sum_{0<k<j} (f| Q_k g) \textnormal{ for } j \in \NN,
\label{eq:Nbj}
\end{equation}
where $Q_j$ is as in \Cref{thm_main_1}.

\begin{main-theorem} \label{thm_main_2} \sl
Let $\beta \in \RR$ be given.
\begin{enumerate}[label=(\alph*)]
\item \label{eq:noconv1} When $m_j \le 2\beta$ there is no $\lambda \in (1, \infty)$ such that the sequence
$$ \{ N^{- \lambda m_j} \mathcal{N}_{\beta, P_j}(f_{N,\lambda},f_{N,\lambda}) : N \in \NN \} $$
converges to $a_j(x_0, \xi_0)$ in probability. In fact, if $m_j < 2\beta$, then for every $\lambda \in (1,\infty)$ and every $ c>0$,
\begin{equation*} 
\lim_{t \to \infty} \PP \left \{  \left |t^{- \lambda m_j} \mathcal N_{\beta,P_j}(f_{t,\lambda},f_{t,\lambda})    - a_j(x_0,\xi_0) \right | >c \right \} = 1.
\end{equation*}
\item \label{eq:noconv2} When $m_j \le 2\beta -1/2$ there is no $\lambda \in (1, \infty)$ such that the sequence
\[ \left\{ \frac{1}{N} \int_{N}^{2N}   t^{- \lambda m_j} \mathcal N_{\beta,P_j}(f_{t,\lambda},f_{t,\lambda})  \,  \dd t : N \in \NN \right\} \]
converges to $a_j(x_0, \xi_0)$ in probability.
Actually, for every $\lambda \in (1,\infty)$ and every $ c>0$,
\begin{equation*} 
\lim_{T \to \infty} \PP \left \{  \left | \frac{1}{T} \int_{T}^{2T}   t^{- \lambda m_j} \mathcal N_{\beta,P_j}(f_{t,\lambda},f_{t,\lambda})  \,  \dd t  - a_j(x_0,\xi_0) \right | >c  \right \} = 1.
\end{equation*}
\end{enumerate}
\end{main-theorem}

Finally, we state the third main result of our paper, which provides the rate of convergence in probability of the limits in \Cref{thm_main_1}.
\begin{main-theorem}\label{thm_main_3} \sl Under the same assumptions of \Cref{thm_main_1} and adopting the notation in \eqref{eq:Nbj}, we have:
\begin{enumerate}[label=(\alph*)]
\item \label{in:non-averaging} Whenever $j_\beta \geq 1 $, set $\lambda_j = \max (1/(m_j - m_{j+1}), 2)$ if $1 \leq j < j_\beta$, and $\lambda_{j_\beta} > \max (1/(m_{j_\beta} - 2\beta), 2)$. Then, for all $\varepsilon > 0$ and $\delta \in (0, 1]$, there exists $N_0$ such that for all $N \geq N_0$
\[\PP \left \{  \left |N^{- \lambda m_j} \mathcal N_{\beta,P_j}(f_{N,\lambda},f_{N,\lambda})    - a_j(x_0,\xi_0) \right | \leq \varepsilon \right \} \geq 1 - \delta.\]
Here $N_0 = C \max (1/\varepsilon, (\log 1/\delta)^{1/\theta})$ with $\theta > 0$ and $C \geq 1$ explicit and only depending on $d$, $\{ m_1, \dots, m_{k_\beta}\}$ and $ \beta $.
\item \label{in:averaging} Set $\lambda_j = \max (1/(m_j - m_{j+1}), 2)$ if $j_\beta < j < k_\beta$, and $\lambda_{k_\beta} > \max (1/(m_{k_\beta} - 2\beta + 1/2), 2)$. Then,
for all $\varepsilon > 0$ and $\delta \in (0, 1]$, there exists $N_0$ such that for all $N \geq N_0$
\[\PP \left \{  \left | \frac{1}{N} \int_{N}^{2N}   t^{- \lambda m_j} \mathcal N_{\beta,P_j}(f_{t,\lambda},f_{t,\lambda})  \,  \dd t  - a_j(x_0,\xi_0) \right | \leq \varepsilon  \right \} \geq 1 - \delta.\]
Here $N_0 = C \max (1/\varepsilon, (\log 1/\delta)^{1/\theta})$ with $\theta > 0$ and $C \geq 1$ explicit and only depending on $d$, $\{ m_1, \dots, m_{k_\beta}\}$ and $ \beta $.
\end{enumerate}
\end{main-theorem}

The proof of the part \ref{lim:averaging} of \Cref{thm_main_1} is based on an ergodicity property for $t^{- \lambda m} \mathcal  E_\beta(\overline{f_{t, \lambda}}, f_{t, \lambda})$ (see the \cref{lemma:ergodic}). To show the ergodicity, we have to compute the covariance of
\[\frac{1}{T} \int_{T}^{2T} t^{- \lambda m} \mathcal  E_\beta(\overline{f_t}, f_t)   \, \dd t,\]
which is reduced to analyse an oscillatory integral and isolate 
appropriately the stationary points.

The problem we study in this article and our results are of interest by themselves. Think for 
example of the fact that one can construct a parametrix for an elliptic differential 
operator in the presence of noise. The problem also has 
connections with quantum mechanics, where the observables are
self-adjoint operators whose symbols are classical observables (functions in the phase space). Furthermore, if a unit vector $\psi$ in a Hilbert space describes the state of a quantum system, the measurement $\langle \psi, P \psi \rangle$
corresponds to the expected value of the classical observable $a$ ---here $P$ is the quantum observable and $a$ its symbol. As the reader can guess, we have borrowed the terms observables, measurements, states, plane waves and wave packets from quantum mechanics. 

Despite these interesting connections,
our motivation originates in the Calder\'on 
problem. This inverse problem consists in recovering the conductivity of a medium from suitable 
electrical superficial data (voltage and current on the surface of the medium). In 
technical terms, the conductivity is described by a positive function $\gamma$ defined in a 
bounded domain $D \subset \RR^d$ with $d \geq 2$. The boundary data is modeled by the so-called 
Dirichlet-to-Neumann map $\Lambda_\gamma$, which is an operator that maps the voltage $f$ on $\partial D$, the boundary of $D$, to the current through the boundary
$\Lambda_\gamma f = \gamma \partial_\nu u|_{\partial D}$. Here $u$ is the electric potential generated by $f$, 
that is, the solution of the Dirichlet problem
\[
\left\{
\begin{aligned}
& \nabla \cdot (\gamma \nabla u) = 0 \textnormal{ in } D,\\
& u|_{\partial D} = f,
\end{aligned}
\right.
\]
and $\partial_\nu$ denotes the partial derivative in the direction given by the exterior 
unit normal vector on $\partial D$. Thus, the Calder\'on problem consists in deciding if 
$\gamma$ is uniquely determined by $\Lambda_\gamma$, and if it is so, reconstructing the 
conductivity from the boundary data. This inverse problem has received a lot of attention 
since Calder\'on formulated it in \cite{zbMATH05684831}. Some important contributions on this 
problem are \cite{zbMATH04015323}, \cite{zbMATH00854849}, \cite{zbMATH}, 
\cite{zbMATH05050053}, \cite{zbMATH06145493} and \cite{zbMATH06534426}.

Since the Dirichlet-to-Neumann map contains physical data about the voltage and the current on the boundary,
one assumes that this information has been obtained after some measurement procedure. In applications, these 
measurements will contain errors. In \cite{CalderonCorrupted}, the authors assumed 
that the available data had been corrupted in the process of measurements so that they only 
have access to
\[ \mathcal{N}(f, g) = \int_{\partial D} \Lambda_\gamma f g \, + \mathcal{E}_0 (f, g),\]
where the Hilbert space involved in the definition of the error $\mathcal{E}_0$ is $L^2(\partial D)$ ---which explains the subindex 
$0$. The authors of \cite{CalderonCorrupted} showed that from the 
measurements $\mathcal{N}$ one can reconstruct $\gamma$ and $\partial_\nu \gamma$ on $\partial D$ almost surely. Their approach 
is based on a method proposed by Brown \cite{zbMATH01731190}, and 
later developed in collaboration with Salo \cite{zbMATH05045048}, to 
reconstruct the unknown on the boundary assuming full knowledge of 
ideal data. This method uses a family of solutions whose boundary 
data remind to wave packets. The formula, given in 
\cite{CalderonCorrupted}, to reconstruct the conductivity on the 
boundary from $\mathcal{N}$ is a limit in the spirit of part 
\ref{lim:non-averaging} in \Cref{thm_main_1}. However, to 
reconstruct its normal derivative on the boundary, the formula 
requires an average in the parameter of the family, as the formula 
in \ref{lim:averaging} of \Cref{thm_main_1}. 

Let's  clarify  in detail the 
connection of \Cref{thm_main_1} with the results in 
\cite{CalderonCorrupted}. First, we notice that, if
$D$ and $\gamma$ are smooth, the map $\Lambda_\gamma$ can be locally 
identified with a pseudodifferential operator of order $m = 1$ whose 
symbol admits an expansion as \eqref{eq:ps_1} with $m_j = 2 - j$. In this case each  $a_j$ contains information of $\partial_\nu^{j - 1} \gamma$ on 
the boundary (this was proved by Sylvester and Uhlmann 
\cite{zbMATH04028037}). Thus, for $\beta = 0$, $m_1 = 1$ is in the 
assumptions of \Cref{thm_main_1} \ref{lim:non-averaging}, so no 
average is needed to reconstruct $a_1$, which contains information 
of $\gamma$ on the boundary of $D$ ---as shown in 
\cite{CalderonCorrupted}. The second order $m_2 = 0$ 
satisfies the conditions of \Cref{thm_main_1} \ref{lim:averaging} 
for $\beta = 0$, so the formula to reconstruct $a_2$ (containing 
information on $\partial_\nu \gamma$ on $\partial D$) requires an 
average ---as in \cite{CalderonCorrupted}. On the other hand, \Cref{thm_main_2} shows that the method 
adopted in \cite{CalderonCorrupted} does not allow to recover $\partial_\nu^k \gamma$ on the boundary for $k > 1$ with $\beta = 0$. 
This point could suggest that is not possible to reconstruct 
$\gamma$ in $D$ from $\mathcal{N}$ almost surely. However, making 
the level of the noise $\mathcal{N}$ small, Abraham and Nickl 
\cite{1906.03486} have 
been able to provide an algorithm to reconstruct $\gamma$ in $D$.

Finally, our results can also be connected to those in 
\cite{Maxwell}, where the authors study the question of data 
corruption in electromagnetism and elasticity. In electrodynamics, 
the boundary data is given by either the admittance map or the 
impedance map. These maps can be identified with pseudodifferential 
operators of order $m = 0$ whose symbols can be expanded as in 
\eqref{eq:ps_1}, containing information of the unknown on the 
boundary (see \cite{mcdowall1997boundary} and 
\cite{joshi2000total}). Thus, according to our results, for $\beta = 0$ we would require making an average to obtain the leading order 
term of the expansion, while for $\beta = -1$ we would not. This 
actually coincides with the formulas given by the authors in 
\cite{Maxwell}. In this regard, the results of our paper establish a framework that helps to obtain a better understanding of the results in \cite{CalderonCorrupted} and \cite{Maxwell}.

The contents of the next sections are the following. In the \cref{sec_deterministic}, we recall the basic definition of classical symbols and show how to use waves packets to recover
the full expansion of the observable. In the \cref{sec_noise}, we define rigorously the random error and prove the main results of this paper, meaning, \Cref{thm_main_1}, \Cref{thm_main_2} and \Cref{thm_main_3}.

Throughout the paper we write $a \lesssim b$ or equivalently $b \gtrsim a$, when $a$ and $b$ are positive constants and there exists $C > 0$ so that $a \leq C b$. We refer to $C$ as the implicit constant and will only depends on $d$, $\{ m_j : j \in \NN \}$ and $\beta$. Additionally, if $a \lesssim b$ and $b \lesssim a$, we write $a \simeq b$.

\section{Recovering an observable from ideal data} \label{sec_deterministic}

In this section we provide an algorithm to recover the expansion
$\sum_{j= 1}^\infty a_j$ of the symbol of an observable $P$ when 
the available data has not been corrupted by measurement errors. The 
advantage of our method is that it can be used to filter the noise up to some extent, as described in \Cref{thm_main_1}.

Let $x_0 $ and $\xi_0 $ belong to $ \RR^d$ with $|\xi_0| =1 $ and take a smooth function $\chi$, such that its Fourier transform $\widehat{\chi}$ has compact  support in $B_1$, the open ball of radius $1$. Consider the \textit{wave packets}
\begin{equation} \label{eq:ft}
f_t(x) =   t^{d/2}\chi(t(x-x_0))e^{it^{\lambda} (x-x_0)\cdot \xi_0} ,
\end{equation}   
where $t \ge 1$ and  $1 < \lambda<\infty$. To reduce the notation in this section, we write $f_t$ instead of $f_{t, \lambda}$. In particular, we choose $\widehat{\chi}$ as a smooth cutoff satisfying
\begin{equation} \label{eq:chi_properties}
b\mathbf{1}_{B_{1/2}}(\xi) \le \widehat{\chi}(\xi) \le b  \mathbf{1}_{B_1}(\xi)  \quad   \forall \xi \in \RR^d \quad \text{and} \quad \norm{\chi} = \norm{\widehat{\chi}} =  1,
 \end{equation}
where $b>0$, and $\mathbf{1}_{B_{1/2}}$ and $\mathbf{1}_{B_1}$ denote the characteristic functions of the balls of radius $1/2$ and $1$, respectively. One can check that the wave packet $f_t$ satisfies
\begin{equation} \label{eq:sob_norm_ft}
\norm{f_t}_\beta \simeq t^{ \lambda  \beta}.
\end{equation}
Heuristically the wave packets $f_t$ will look increasingly like a plane wave close to the point $x_0$.  This is the reason to introduce the parameter $\lambda > 1$, since it guarantees that the frequency of the wave increases more quickly than the support of $\chi(t(\centerdot - x_0))$  shrinks around $x_0$. 

Before providing the algorithm to recover the expansion of the observable $P$, let us recall a few definitions. We say that a smooth function $a = a(x,\xi)$ in the \textit{phase space} $\RR^d \times \RR^d$ is a \textit{classical symbol} if it satisfies that
 \[
|\partial_x^\alpha \partial_\xi^\beta a(x,\xi)| \le C_{\alpha,\beta} \br{\xi}^{m-|\beta|},
 \]
where $\alpha $ and $ \beta$ are multi-indices and $\br{\xi} = (1+|\xi|^2)^{1/2}$. Additionally, if $\{ m_j : j \in \NN \}$ is a strictly decreasing sequence of real numbers so that $\lim_{j \to \infty} m_j = -\infty$, and $a_j$ is a classical symbol of order $m_j$, we say that $\sum_{j = 1}^\infty a_j$ expands the symbol $a$ of order $m = m_1$, and we write
$a \sim \sum_{j = 1}^\infty a_j$, if
\[a - \sum_{j = 1}^k a_j \textnormal{ with } k\in \NN \]
is a classical symbol of order $m_{k+1}$. Notice that, as remarked in the introduction, the expansion of $a$ is non unique, by the definition, any $a_j$ can be modified by adding to it a lower order symbol. On the other hand, if the homogeneity condition \eqref{eq:ps_2} is added, then $a_j(x,\xi)$ is uniquely determined for $|\xi| \ge 1/2$ ---there are no lower order symbols with the same homogeneity of $a_j$--- but it still can be modified freely for $|\xi| \le 1/2$.

We are now ready to describe the algorithm to recover the homogeneous 
expansion of a given observable. However, for simplicity, we start by  
assuming that the observable $P$
has a symbol of the form $a + b$ with $a$ and $b$ being, respectively, classical symbols of order $m$ and $n$ such that, $a$ is homogeneous in the variable $\xi$, and $n < m$. The objective is to   show that the homogeneous part of the  symbol can be recovered  from knowledge of $(f_t|P f_t)$, where $f_t$ is the family of wave packets given in \eqref{eq:ft}.
We want to prove that
\begin{equation} 
\lim_{t \to \infty} t^{-\lambda m}(f_t| P f_t) =  a(x_0,\xi_0) ,
\label{lim:claim}
\end{equation}   
with $1<\lambda<\infty$. 

We begin with a simple lemma that will be used to estimate the part of the observable $P$ corresponding to $b$.
\begin{lemma} \label{lemma_vanishing} \sl
Let $R$ be a pseudodifferential  operator of order  $n \in \RR$. Then 
\[  |(f_t| R f_t)| \lesssim t^{\lambda n}.\]
\end{lemma}
\begin{proof}
 Let $b$ denote the symbol of $R$. Note that we have
\[
(f_t| R(f_t) )   =  \frac{1}{(2\pi)^{d/2}} \int_{\RR^d} \overline{f_t(x)} \Big( \int_{\RR^d} e^{i x \cdot \xi} b(x,\xi)  \widehat{f_t}(\xi) \, \dd \xi \Big) \,\dd x .
\]
By the basic properties of the Fourier transform and the form of \eqref{eq:ft} we have that 
\begin{equation} \label{eq:ft_fourier}
 \widehat{f_t}(\xi) = t^{-d/2}  e^{-i \xi \cdot x_0} \widehat{\chi}((\xi-t^{\lambda}\xi_0)/t).
\end{equation}  
After the changes of variables
\begin{equation} \label{eq:change_v}
\eta = (\xi-t^{\lambda}\xi_0)/t \quad \text{and} \quad y= t(x-x_0),
\end{equation}
we can ensure that
\begin{equation} 
|(f_t| R(f_t) )| 
\lesssim  \int_{\RR^d}  |\chi(y)|  \int_{\RR^d} \left |b(x_0 + t^{-1}y,t^{\lambda}\xi_0 + t \eta) \right|   |  \widehat{\chi}(\eta) | \, d\eta \, dy .
\label{in:symbol}
\end{equation}
Since $b$ a classical symbol of order $n$, we know that $|b(x,\xi) | \lesssim \br{\xi}^n$ for all $x, \xi \in \RR^d$, in particular
\[  | b(x_0 + t^{-1}y,t^{\lambda}\xi_0 + t \eta) | \lesssim \br{t^{\lambda}\xi_0 + t \eta}^n.\]
Note now that whenever $\eta \in \supp \widehat{\chi} \subset B_1 $, we have that $\br{t^{\lambda}\xi_0 + t \eta}^n \simeq t^{\lambda n}$. This equivalence together with \eqref{in:symbol} finishes the proof of this lemma.
\end{proof}

We now prove that the limit \eqref{lim:claim} holds.

\begin{proposition} \label{prop:a} \sl
Let  $P$ be a pseudodifferential operator with a symbol of the form $a + b $ with $a$ and $b$ classical symbols of order $ m$ and $ n < m$, respectively. Additionally, assume that $a$ is homogeneous in the variable $\xi$. Consider the wave packets $\{ f_t : t \geq 1 \}$ defined in \eqref{eq:ft}.
Then,
\[ 
  \left |t^{- \lambda {m}} (f_t| P f_t) - a(x_0,\xi_0)\right | =  O(t^{-1})  + O(t^{\lambda(n - m)}) + O(t^{1- \lambda}),
\]
and consequently
\[
  \lim_{t \to \infty} t^{- \lambda {m}} (f_t| P f_t) = a(x_0,\xi_0).
\]
\end{proposition}
\begin{proof}
Let $Q$ and $R$ denote the pseudodifferential operators with symbols $a$ and $b$, respectively. Thus, $P = Q + R$. By the \cref{lemma_vanishing} we have that  
\[ t^{ - \lambda m}(f_t| R f_t ) = O(t^{\lambda(n - m)}).\]
Hence, it is enough to show that
\begin{equation} \label{eq:conv_1}
t^{- \lambda {m}}(f_t| {Q} f_t)  =  t^{- \lambda {m}}(f_t|a(\cdot,t^{\lambda}\xi_0)f_t) +  O(t^{1-\lambda}),
\end{equation}
since, using that $a(\cdot,t^{\lambda}\xi_0) = t^{\lambda {m}} a(\cdot,\xi_0)$, one has
\begin{multline*}
\left |t^{- \lambda {m}}(f_t|a(\cdot,t^{\lambda}\xi_0)f_t) - a(x_0,\xi_0) \right | = \left |\int_{\RR^d} \overline{f_t(x)}  a(x,\xi_0)f_t(x) \, \dd x -  a(x_0,\xi_0) \right | \\
\le \int_{\RR^d} t^{d} |\chi(t(x-x_0))|^2  |a(x,\xi_0)- a(x_0,\xi_0)| \,\dd x = O(t^{-1}), 
\end{multline*}
using that $a (\centerdot, \xi_0)$ is a Lipschitz function with uniform bound for $|\xi_0| = 1$. Thus, we just need to prove \eqref{eq:conv_1}. Note that using the Leibniz rule this is immediate if $Q$ is a differential operator. Since $Q$ has $a$ as symbol, we can write
\begin{multline*}
(f_t| {Q} f_t) - (f_t|a(\cdot,t^{\lambda}\xi_0)f_t)   \\ = \frac{1}{(2\pi)^{d/2}} \int_{\RR^d} \overline{f_t(x)} \Big( \int_{\RR^d} e^{i x \cdot \xi} \left(a(x,\xi)- a(x,t^{\lambda}\xi_0) \right)  \widehat{f_t}(\xi) \, \dd\xi \Big) \, \dd x
\end{multline*}
using the inversion formula of the Fourier transform.
We now use the form of $\widehat{f_t}$ given in \eqref{eq:ft_fourier}   and the changes of variables \eqref{eq:change_v}. Thus,
\begin{equation}
\begin{aligned} 
& \left |(f_t| {Q} f_t ) - (f_t|a(\cdot,t^{\lambda}\xi_0)f_t) \right |  \\ 
&\le     \int_{\RR^d}  |\chi(y)|  \int_{\RR^d}   \left |a(x_0 + t^{-1}y,t^{\lambda}\xi_0 + t \eta)- a(x_0 + t^{-1}y,t^{\lambda}\xi_0) \right|   |  \widehat{\chi}(\eta) | \, \dd\eta \, \dd y.
\end{aligned}
\label{eq:conv_2}
\end{equation}
Now, since $\widehat{\chi}$ is supported in the ball of radius 1, we can use that for $|\eta| \le 1$, one has
\begin{multline*}
 t^{- \lambda {m}} \left |a(x_0 + t^{-1}y,t^{\lambda}\xi_0 + t \eta)- a(x_0 + t^{-1}y,t^{\lambda}\xi_0) \right|   \\ 
=   \left |a( x_0 + t^{-1}y,\xi_0 + t^{1-\lambda} \eta )- a(x_0 + t^{-1}y,\xi_0 ) \right|  \\
\le t^{1-\lambda} \sup_{\zeta \in B_1}|\nabla_\xi a(x_0 + t^{-1}y,  \xi_0 + \zeta) | = O( t^{1-\lambda}).
\end{multline*}
Note that we have used the homogeneity property to get the second line.  Multiplying \eqref{eq:conv_2} by $t^{- \lambda m}$ and using the previous inequality we obtain the desired result.
\end{proof}

Finally, we use the \cref{prop:a} to derive the algorithm to recover the full expansion of the observable $P$ from the ideal data. Recall that its symbol satisfies
\[a \sim \sum_{j = 1}^\infty a_j\]
with $a_j$ of order $m_j$ and homogeneous in the variable $\xi$. For convenience, set $P_1 = P$ and note that its symbol can be written as
$a = a_1 + b_1$ with $a_1$ homogeneous of order $m_1 = m$ and  $b_1 = a - a_1$ of order $m_2 < m_1$. By the \cref{prop:a},
\begin{equation}
  \left |t^{- \lambda {m_1}} (f_t| P_1 f_t) - a_1(x_0,\xi_0)\right | =  O(t^{-1})  + O(t^{\lambda(m_2 - m_1)}) + O(t^{1- \lambda}),
\label{lim:P1}
\end{equation}
and consequently
\begin{equation}
  \lim_{t \to \infty} t^{- \lambda {m_1}} (f_t| P f_t) = a_1(x_0,\xi_0).
  \label{lim:a_1}
\end{equation}
Next, let $Q_1$ denote the pseudodifferential operator with symbol $a_1$, and set $P_2 = P - Q_1$. The symbol of $P_2$ can be written as $a_2 + b_2$ with $a_2$ homogeneous of order $m_2$ and $b_2 = a - a_1 - a_2$ of order $m_3 < m_2$. Again, by the \cref{prop:a} we have that
\[ 
  \left |t^{- \lambda {m_2}} (f_t| P_2 f_t) - a_2(x_0,\xi_0)\right | =  O(t^{-1})  + O(t^{\lambda(m_3 - m_2)}) + O(t^{1- \lambda}),
\]
and consequently
\[
  \lim_{t \to \infty} t^{- \lambda {m_2}} \big[ (f_t| P f_t) - (f_t| Q_1 f_t) \big] = a_2(x_0,\xi_0).
\]
Now, assume that we have already recovered $a_1, \dots, a_{j-1}$ and let
$Q_1, \dots, Q_{j-1}$ denote the pseudodifferential operators with symbols $a_1, \dots, a_{j-1}$, respectively. Then, set
\begin{equation}
P_j = P - \sum_{k = 1}^{j-1} Q_k, \textnormal{ for } j\in \NN \textnormal{ with } j > 1.
\label{id:Pj}
\end{equation}
Note that the symbol of $P_{j}$ can be written as $a_{j} + b_{j}$ with $a_{j}$ homogeneous of order $m_{j}$ and $b_{j} = a - \sum_{k = 1}^{j} a_k $ of order $m_{j+1} < m_{j}$. By the \cref{prop:a} we have that
\begin{equation}
  \left |t^{- \lambda {m_j}} (f_t| P_j f_t) - a_j(x_0,\xi_0)\right | =  O(t^{-1})  + O(t^{\lambda(m_{j+1} - m_j)}) + O(t^{1- \lambda}),
\label{lim:Pj}
\end{equation}
and consequently
\begin{equation}
\lim_{t \to \infty} t^{- \lambda {m_j}} \big[ (f_t| P f_t) - \sum_{k=1}^{j-1} (f_t| Q_k f_t) \big] = a_j(x_0,\xi_0).
  \label{lim:aj}
\end{equation}
Note that this algorithm only requires knowledge of $P$ and the different orders $\{ m_j : j\in \NN \}$ of the homogeneous expansion of its symbol.

\section{Filtering the noise in the measurements} \label{sec_noise}
In this section we show how to filter out the random noise to recover the observable from 
the measurements with full certainty. Unfortunately, ---as described in \Cref{thm_main_2}--- with our method it is impossible to 
recover the full expansion of the symbol of $P$ since, for very low order terms of the 
expansion, the variance of the error becomes so large that the signal produced by 
the observable is lost in the noise. We also show  in this section how this phenomenon works.

Before addressing these questions, we recall some properties about white noise, and define 
the random errors contained in our measurements. Let $(\Omega, \mathcal{F}, \PP)$ be a probability space and $\mathcal{H}$ be a Hilbert space. A linear map $\WW : \mathcal{H} \longrightarrow L^2(\Omega, \mathcal{F}, \PP) $ is said to be a \textit{complex Gaussian white noise} if, for every $f \in \mathcal{H}$, $\WW f$ is a centered complex Gaussian variable and
\begin{equation}
\EE (\overline{\WW f}\WW g) = (f|g)_\mathcal{H}, \textnormal{ for all } f,g \in \mathcal{H}.
\label{id:isometry}
\end{equation}
Here $(\centerdot|\centerdot)_\mathcal{H}$ denotes the inner product of $\mathcal{H}$. If $\mathcal{H}$ is a separable Hilbert space and $(\Omega, \mathcal{F}, \PP)$ is a probability space so that there exists a sequence $\{ X_n : n \in \NN \}$ of independent complex Gaussian so that
\[\EE X_n = 0, \qquad \EE (\overline{X}_n X_n) = 1,\]
then, there always exists a complex Gaussian white noise for $(\Omega, \mathcal{F}, \PP)$ and $\mathcal{H}$. It is enough to consider an 
orthonormal basis $\{ e_n : n \in \NN \}$ of $\mathcal{H}$, and define
$\WW f = \sum_{n\in\NN} (e_n|f)_\mathcal{H} X_n$, where the convergence takes place in
$L^2(\Omega, \mathcal{F}, \PP)$. On 
the other hand, if $\WW$ is a complex Gaussian white noise with $\mathcal{H}$ separable, we
set $X_n = \WW e_n$, and then $\WW f = \sum_{n\in \NN} (e_n|f)_\mathcal{H} X_n$. By definition $X_n$ is a centered complex Gaussian variable and since
$\WW$ is an isometry, $\EE (\overline{X}_n X_n) = 1$. Additionally,
$X_1, \dots, X_n, \dots$ are independent, since they are uncorrelated and every finite 
linear combination of any of them is a centered Gaussian variable.

With these properties at hand, we define now the \textit{error} $\mathcal{E}_\beta$ 
from a complex Gaussian white noise in the Hilbert space
$\mathcal{H} = H^\beta(\RR^d) \otimes H^\beta(\RR^d)$ as follows:
\begin{equation}
\mathcal{E}_\beta (f, g) = \WW (f \otimes g).
\label{def:error}
\end{equation}
Recall that $(f_1\otimes g_1 | f_2\otimes g_2)_\mathcal{H} = (f_1|f_2)_\beta (g_1|g_2)_\beta$.
This explains why the error $\mathcal{E}_\beta$ takes the form in \eqref{id:error}, 
with $X_\alpha = \WW (e_{\alpha_1} \otimes e_{\alpha_2})$ and $\{ e_n: n \in \NN \}$ 
an orthonormal basis of $H^\beta(\RR^d)$. In \cite{CalderonCorrupted} and \cite{Maxwell}, the authors chose the form of \eqref{id:error} to define the error. Here we have chosen a more intrinsic way of defining it. After these comments, it is straightforward to check \eqref{eq:varianza}: $\EE \mathcal{E}_\beta (f, g) = \EE \WW (f \otimes g) = 0$, and
\[\EE |\mathcal{E}_\beta (f, g)|^2 = \EE |\WW (f \otimes g)|^2 = (f \otimes g|f \otimes g)_\mathcal{H} = (f|f)_\beta (g|g)_\beta = \norm{f}_\beta^2 \norm{g}_\beta^2. \]

Note that filtering out the noise, consists essentially in proving that the error $\mathcal{E}_\beta (\overline{f_t}, f_t)$ goes to zero almost surely when $t$ tends to infinity. This property can be seen from its variance $\EE | \mathcal{E}_\beta (\overline{f_t}, f_t) |^2$, and this is the reason to state here the following lemma.

\begin{lemma} \label{prop:as_convergence} \sl
Let $(X,\Sigma,\mu)$ be a measure space, and let $\{f_n: n \in \NN \}$ be a sequence in $L^p(X,\Sigma,\mu)$ with $1 \le p< \infty$ converging to $f$ in $L^p(X,\Sigma,\mu)$. Assume there exists a sequence  $\{  b_n : n\in \NN \}$ of positive real numbers whose limit vanishes, and 
\begin{equation} \label{eq:series}
 \sum_{n=1}^\infty \frac{1}{b_n^p} \int_{X} |f_n-f|^p  \, d\mu  <\infty. 
\end{equation}
Then, $f_n(x) $ goes to $ f(x)$ as $n  $ tends to $ \infty $ for almost every $x\in X$.
\end{lemma}
This fact is well known, but for completeness, we include a proof.
\begin{proof}
Set $E_n = \{x\in X: |f_n(x) -f(x) | > b_n \}$. Using Chebyshev's inequality we see that
\[ \mu \left( \bigcap_{n=1}^\infty \bigcup_{k=n}^\infty E_k \right) \le \sum_{k=n}^{\infty} \mu(E_k) \le \sum_{k=n}^{\infty} \frac{1}{b_n^p} \int_{X} |f_n-f|^p \,  d\mu   ,\]
which, after \eqref{eq:series},  yields that
\[
 \mu \left( \bigcap_{n=1}^\infty \bigcup_{k=n}^\infty E_k \right) = 0.
 \]
On the other hand, if $x\notin \cap_{n=1}^\infty \cup_{k=n}^\infty E_k$, it means that 
there exists an $n_x$ such that $|f_n(x) - f(x)  | \le b_n $ for all $n \ge n_x$. Thus, 
$f_n (x) $ tends to $ f (x)$ for almost every $x \in X$.
\end{proof}

Now, we are in the position to prove the theorems stated in the introduction of this paper. 
We start by showing when it is possible to use wave packets to filter the noise.

\begin{proof}[Proof of \cref{thm_main_1}]
Start by noting that, whenever $1 < \lambda < \infty$, $\beta \in \RR$ and $j \in \NN$, we have by \eqref{eq:varianza} and \eqref{eq:sob_norm_ft} that
\begin{equation}
\EE |t^{-\lambda m_j} \mathcal{E}_\beta (\overline{f_t}, f_t)|^2 \simeq t^{-2\lambda (m_j - 2\beta)},
\label{eq:error}
\end{equation}
where the implicit constant only depends on $d$ $\lambda $, $m_j$ and $\beta$.
Then, we have,  by \eqref{lim:P1} and \eqref{eq:error}, that
\[\EE |t^{-\lambda_1 m_1} \mathcal{N}_{\beta, P} (f_{t,\lambda_1}, f_{t,\lambda_1}) - a_1 (x_0, \xi_0) |^2 \lesssim t^{-2}. \]
While for $j \in \{ 2, \dots, j_{\beta} \}$, we have, by \eqref{lim:Pj} and \eqref{eq:error}, that 
\[\EE \Big|t^{-\lambda_j m_j} \big[\mathcal{N}_{\beta, P} (f_{t,\lambda_j}, f_{t,\lambda_j}) - \sum_{k=1}^{j-1} (f_{t,\lambda_j}| Q_k f_{t,\lambda_j}) \big] - a_j (x_0, \xi_0) \Big|^2 \lesssim t^{-2}. \]
Thus, to prove the part \ref{lim:non-averaging} of \cref{thm_main_1}, we apply the \cref{prop:as_convergence} with $b_n = n^{-1/2 + \varepsilon}$ with $\varepsilon < 1/2$.

Next we continue with the part \ref{lim:averaging} of \cref{thm_main_1}. Note that for $j > j_\beta$ the variance of the error does not decay any more, it could even grow ---see \eqref{eq:error}. Thus, inspired by the strong law of large numbers or, more generally, basic ergodicity results, we perform an average on the parameter $t$ of the family of wave packets. This motivates the next lemma, which is a refinement of \cite[Lemma 3.5]{CalderonCorrupted}.

\begin{lemma} \label{lemma:ergodic} \sl
Consider $1<\lambda< \infty $ and $m ,\beta \in \RR$. Assume also that $x_0,\xi_0 \in \RR^d$ with $|\xi_0|=1$ in the definition of the wave packets $f_t$. Then, for $T > 2^{1/(\lambda-1)}$
\[ \EE \left | \frac{1}{T} \int_{T}^{2T} t^{- \lambda m} \mathcal  E_\beta(\overline{f_t}, f_t)   \, \dd t \, \right|^2  \simeq  T^{2\lambda (2\beta -\frac{1}{2}-m) +1 }.\]
where the implicit constants in the equivalence  depend on  $d$, $m$, $\lambda$ and $\beta$.
\end{lemma}
\begin{proof}
For convenience, set $Q_T = [T,2T] \times [T,2T]$. Note that by \eqref{def:error} and \eqref{id:isometry}
\[ \EE \big[ \mathcal  E_\beta(\overline{f_t}, f_t) \overline{\mathcal  E_\beta(\overline{f_s}, f_s)} \big] = |(f_t |f_s)_\beta|^2, \]
and consequently,
\begin{multline} \label{eq:est_sim}
\EE \left | \frac{1}{T} \int_{T}^{2T} t^{- \lambda m} \mathcal  E_\beta(\overline{f_t}, f_t) \, \dd t \, \right|^2 \\
= \frac{1}{T^{2}} \int_{Q_T} t^{- \lambda m }s^{- \lambda m} |(f_{t}|f_{s})_\beta |^2 \, \dd (t,s)
\simeq  T^{-2 \lambda m-2} \int_{Q_T}  |(f_t |f_s)_\beta|^2 \, \dd (t,s). 
\end{multline}
We claim that 
 \begin{equation} \label{eq:claim}
 (f_t |f_s)_\beta  \simeq T^{2 \lambda \beta} (f_t |f_s) \textnormal{ for all } (t,s) \in Q_T.
 \end{equation}
To see this we plug in $\widehat{f_t}(\xi) = t^{-d/2}  e^{-i \xi \cdot x_0} \widehat{\chi}((\xi-t^{\lambda}\xi_0)/t)$ to \eqref{eq:inner_p} and recall that $\widehat{\chi}$ is non-negative and supported in the ball $B_1$. The change of variables $\eta =(\xi-s^{\lambda}\xi_0)/s $ yields
\begin{multline}\label{eq:beta_sim}
 |(f_t |f_s)_\beta| =  t^{-d/2} s^{-d/2}  \int_{\RR^d}\widehat{\chi}\left( \frac{\xi-t^{\lambda}\xi_0}{t} \right )  \widehat{\chi}\left( \frac{\xi-s^{\lambda}\xi_0}{s} \right )  \br{\xi}^{2\beta} \, \dd \xi   \\
 =  t^{-d/2} s^{d/2}   \int_{B_1}  { \widehat{\chi} \left (\frac{s}{t} \eta + \frac{s^\lambda-t^\lambda}{t}   \xi_0  \right ) } \widehat{\chi}( \eta)    \br{s\eta + s^{\lambda} \xi_0}^{2\beta} \, \dd\eta   \simeq  s^{2 \lambda \beta} |(f_t |f_s)|.
\end{multline}
To obtain the last line we have used that $1/2 s^{\lambda} \le \br{s\eta + s^{\lambda} \xi_0} $ for $|\eta| \le 1$ and $s \geq 2^{1/(\lambda-1)}$, and consequently that $\br{s\eta + s^{\lambda} \xi_0} \simeq s^{\lambda}$. This proves the claim. Then, by \eqref{eq:est_sim} and \eqref{eq:claim}, to finish the proof of  the lemma it is enough to show that
 \begin{equation} \label{eq:prod_sim}
\int_{Q_T} |(f_t |f_s)|^2 \, \dd (t,s)  \simeq  T^{3-\lambda}.
 \end{equation}
We begin with the upper bound. Since
\begin{equation} \label{eq:prod_identity}
(f_t |f_s) = t^{d/2} s^{d/2} \int_{\RR^d} e^{i(t^\lambda -s^\lambda)  x\cdot\xi_0} \chi(t(x-x_0)) \chi(s(x-x_0))\, \dd x ,
\end{equation}
  to exploit the cancellations  produced by the exponential factor, let $L$ be the formally self-adjoint operator given by $L f = t+s - i\frac{t-s}{|t-s|}\xi_0 \cdot \nabla f $. Then, since $|\xi_0| =1$, we have that
\begin{align*}
 (f_t |f_s) &=  \frac{t^{d/2} s^{d/2}}{ t+s +|t^\lambda-s^\lambda|} \int_{\RR^d} L(e^{i(t^\lambda-s^\lambda)  x\cdot\xi_0}) \chi(t(x-x_0)) \chi(s(x-x_0))\, \dd x \\
 &=   \frac{t^{d/2} s^{d/2}}{ t+s +|t^\lambda-s^\lambda|} \int_{\RR^d} e^{i(t^\lambda-s^\lambda)  x\cdot\xi_0} L \left(\chi(t(x-x_0)) \chi(s(x-x_0)) \right) \, \dd x,
 \end{align*}
and using that
 \begin{multline*} |L \left (\chi(t(x-x_0)) \chi(s(x-x_0))  \right) |  \le (t+s)  |\chi(t(x-x_0)) \chi(s(x-x_0))|   \\ + t |\nabla\chi(t(x-x_0)) \chi(s(x-x_0))|+ s|\chi (t(x-x_0)) \nabla\chi(s(x-x_0))| ,
 \end{multline*}
leads, since $t+s \simeq T$ whenever $(t,s) \in Q_T$, to
 \begin{equation*}
 |(f_t |f_s)| \lesssim   \frac{t+s }{ t+s +|t^\lambda-s^\lambda| }  \norm{\chi}_1^2 \lesssim  \frac{1}{1+ T^{-1}{|t^\lambda-s^\lambda|}}.
 \end{equation*}
Moreover, whenever $(t,s) \in Q_T$, we have
\begin{equation*}
 |(f_t |f_s)|^2  \lesssim T^{1 - \lambda} \frac{s^{\lambda - 1}}{(1+ T^{-1}{|t^\lambda-s^\lambda|})^2}.
 \end{equation*}
Now, we make the change of variables $u=t$, $v =  T^{-1}(s^\lambda -t^{\lambda}) $ where the volume form is  $\dd u \, \dd v = \lambda T^{-1} s^{\lambda-1}\dd (t,s)$, and therefore
 \begin{multline} \label{eq:Tk1}
 \int_{Q_T} |(f_t |f_s)|^2 \, \dd (t,s) \lesssim T^{1 - \lambda}  \int_{Q_T}  \frac{s^{\lambda - 1}}{(1+ T^{-1}{|t^\lambda-s^\lambda|})^2}  \, \dd (t,s) \\
  \lesssim T^{2-\lambda} \int_T^{2T} \, \dd u \int^\infty_{-\infty}  \frac{1}{(1+|v|)^{2}}  \, \dd v \simeq T^{3-\lambda}.
\end{multline}
 This finishes the proof of the upper bound, next we prove the lower bound. Using the second identity in \eqref{eq:beta_sim} with $\beta =0$, we have
\begin{align*}
 |(f_t |f_s)| &=   t^{-d/2} s^{d/2}   \int_{B_1}   {\widehat{\chi} \left (\frac{s}{t} \eta + \frac{s^\lambda-t^\lambda}{t}   \xi_0   \right )} \widehat{\chi}( \eta)    \, \dd\eta  \\
  &\ge   b t^{d/2} s^{-d/2} \int_{B_{1/2}}   \widehat{\chi} \left (\frac{s}{t} \eta + \frac{s^\lambda-t^\lambda}{t}   \xi_0   \right )   \, \dd\eta ,
 \end{align*}
 where to get the last line we have used that by \eqref{eq:chi_properties}, $\widehat{\chi}$ is nonnegative and  $\widehat{\chi}(\eta)=  b $ in $B_{1/2}$. This also implies that the last integral can be bounded below by $b$ times the measure of the set $B_{1/2} \cap B_{s,t}$ where $B_{s,t}$ is the ball of radius {$r=t/(2s)$} and center  {$c = (t^\lambda-s^\lambda)/s \,  \xi_0  $}. 
Whenever $(t,s) \in Q_T$, we have that $r \ge  1/4$ and $|c| \le  T^{-1}|t^{\lambda}-s^{\lambda}| $. If additionally, $ (t, s) \in Q_T$ is so that $  |t^{\lambda}-s^{\lambda}| \le T/4$, then $|c| \le 1/4$. This implies that actually we have at least $ |B_{1/2} \cap B_{s,t}| \ge |B_{1/4}|$ since the previous estimates for $r$ and $c$ imply that the intersection must contain  a ball of radius $1/4$. Therefore, we have
\[|(f_t |f_s)|  \ge b^2 t^{d/2} s^{-d/2} |B_{1/2} \cap B_{s,t}| \ge  b^2 2^{-d/2}|B_{1/4}|,\]
in the set $D= \{ (t, s) \in Q_T : |t^{\lambda}-s^{\lambda}| \le T/4\}$, and consequently that
\[ \int_{Q_T} |(f_t |f_s)|^2 \, \dd (t,s) \gtrsim \int_{D}  \, \dd (t,s) \simeq T^{1 - \lambda} \int_{D} s^{\lambda - 1} \, \dd (t,s) \simeq T^{2 - \lambda} \int_T^{2T} \, \dd u \int^{1/4}_{-1/4}  \, \dd v,\]
by the same change of variables used in \eqref{eq:Tk1}. This proves the lower bound in \eqref{eq:prod_sim}, which ends the proof of the \cref{lemma:ergodic}.
\end{proof}
By \eqref{lim:Pj} and the \cref{lemma:ergodic}, we have that
\[\EE \left| \frac{1}{T} \int_T^{2T} t^{-\lambda_j m_j} \big[\mathcal{N}_{\beta, P} (f_{t, \lambda_j}, f_{t, \lambda_j}) - \sum_{0<k<j} (f_{t, \lambda_j}| Q_k f_{t, \lambda_j}) \big] \, \dd t - a_j (x_0, \xi_0) \right|^2 \lesssim T^{-1-2\epsilon_j} \]
for certain $\epsilon_j > 0$. Thus, the part \ref{lim:averaging} of \cref{thm_main_1} follows from the \cref{prop:as_convergence} with $b_n = n^{-\epsilon_j/2}$. This finishes the proof of \Cref{thm_main_1}.
\end{proof}

Next, we show why wave packets are not suitable states to recover the 
lower order terms of the expansion of the observable.
\begin{proof}[Proof of \Cref{thm_main_2}]
In the view of \eqref{eq:Nbj} and the limits \eqref{lim:a_1} and \eqref{lim:aj}, one realizes that the relevant quantities
to prove \ref{eq:noconv1} and \ref{eq:noconv2}  are the random variables given by the error and the averaged error produced by the wave packets:
\begin{equation} \label{eq:XY}
 X_t  =  t^{- \lambda m} \mathcal  E_\beta(\overline{f_t}, f_t),   \qquad
 Y_T  = \frac{1}{T} \int_{T}^{2T} t^{- \lambda m} \mathcal  E_\beta(\overline{f_t}, f_t)   \, \dd t.
 \end{equation} 
Here we simplify the notation as we did in the \cref{sec_deterministic}, we write $f_t$ instead of $f_{t,\lambda}$ for the states.

The variables $X_t$ and $Y_T$ are complex Gaussian with zero mean. The variable $X_t$ is complex Gaussian because $\WW (\overline{f_t} \otimes f_t)$ is a complex Gaussian variable. To see that $Y_T$ is complex Gaussian, we should understand the integral in $\dd t$ as a limit of finite linear combinations $\{ \mathcal  E_\beta(\overline{f_{t_1}}, f_{t_1}), \dots, \mathcal  E_\beta(\overline{f_{t_n}}, f_{t_n}) \}$ for every $n \in \NN$. These linear combinations are actually complex Gaussian variables by the definition and the linearity of $\WW$. Eventually, the limit of of complex Gaussian variables is complex Gaussian, so consequently $Y_T$ is a complex Gaussian variable.

Recall that if $Z$ is a complex Gaussian variable with zero mean 
we have that
\begin{equation} \label{eq:prob_formula}
\PP \{ |Z| > c \} = \frac{1}{2\pi |\textnormal{det} K|^{1/2}} \int_{|z| > c} \exp(-z \cdot K^{-1} z/2) \, \dd z ,
\end{equation}
where $K$
is the covariance matrix {of Z} given by
\[ 
K = \frac{1}{2}
\begin{bmatrix}
\EE |Z|^2 + \Re(\EE Z^2) & \Im (\EE Z^2)\\
\Im (\EE Z^2) & \EE |Z|^2 - \Re (\EE Z^2)
\end{bmatrix}.
\]
In the special case where $\EE |Z|^2 = \sigma^2$ and $\EE Z^2 = 0$, one can check ---by computing explicitly the integral \eqref{eq:prob_formula}--- that its distribution is generated by
\begin{equation} \label{eq:dist}
 \PP \{ |Z|>c \} = e^{-c^2/\sigma^2}.
 \end{equation}
We will see in the next lemma that $\EE X_t^2 = \EE Y_T^2 = 0$, and therefore, the distributions of $X_t$ and $Y_T$ are generated in the same way.
{

\begin{lemma} \label{lemma:prod} \sl
Consider $1<\lambda< \infty $ and $m ,\beta \in \RR$. Assume also that $x_0,\xi_0 \in \RR^d$ with   $|\xi_0|=1$ in the definition of the wave packets $f_t$. Then, for any  $t,T > 1$ we have
\begin{equation} \label{eq:prod1}
 \EE \left ( t^{- \lambda m} \mathcal  E_\beta(\overline{f_t}, f_t)  \right)^2  = 0.
 \end{equation}
and
\begin{equation} \label{eq:prod2}
 \EE \left ( \frac{1}{T} \int_{T}^{2T} t^{- \lambda m} \mathcal  E_\beta(\overline{f_t}, f_t)   \, \dd t \, \right)^2  = 0.
 \end{equation}
\end{lemma}
\begin{proof}
We start with the proof of \eqref{eq:prod2}.  By  \eqref{def:error} and \eqref{id:isometry}
\begin{multline} \label{eq:prodH}
\EE \left( \mathcal  E_\beta ( \overline{f_t}, f_t) \mathcal  E_\beta ( \overline{f_s}, f_s)  \right) = \EE \left( \WW(\overline{f_t} \otimes {f_t} ) \WW( \overline{f_s} \otimes {f_s} )  \right) \\
  = ({f_t} \otimes \overline{f_t}|  \overline{f_s} \otimes{f_s} )_{\mathcal H} = ({f_t}|\overline{f_s} )_\beta (\overline{f_t}|{f_s} )_\beta.
 \end{multline}
On the one hand, using the notation of the \cref{lemma:ergodic}
\begin{multline} \label{eq:est_sim2}
\EE \left ( \frac{1}{T} \int_{T}^{2T} t^{- \lambda m} \mathcal  E_\beta(\overline{f_t}, f_t) \, \dd t \, \right)^2 \\
= \frac{1}{T^{2}} \int_{Q_T} t^{- \lambda m }s^{- \lambda m} \EE \left( \mathcal  E_\beta ( \overline{f_t}, f_t) \mathcal  E_\beta ( \overline{f_s}, f_s)  \right)  \, \dd (t,s) \\
= \frac{1}{T^{2}} \int_{Q_T} t^{- \lambda m } s^{- \lambda m} (\overline{f_t}|{f_s} )_\beta({f_t}|\overline{f_s} )_\beta  \, \dd (t,s),
\end{multline}
 Furthermore, note that 
\begin{equation} \label{eq:prodH2}
(\overline{f_t}|{f_s} )_\beta  =   t^{-d/2} s^{-d/2}   \int_{\RR^d}   \widehat{\chi}\left( \frac{-\xi-t^{\lambda}\xi_0}{t} \right )  \widehat{\chi}\left( \frac{\xi-s^{\lambda}\xi_0}{s} \right )  \br{\xi}^{2\beta} \, \dd \xi =0,
\end{equation}
since the supports of the two smooth cut-off functions do not overlap for any $s,t>1$, $|\xi_0| = 1$, and $1<\lambda<\infty$. Indeed, the function $\xi \mapsto \widehat{\chi}\big( -(\xi+t^{\lambda}\xi_0)/t \big) $  is supported in a ball of radius $t$ centered at $- t^\lambda \xi_0$ while the function $\xi \mapsto \widehat{\chi}\big( (\xi-s^{\lambda}\xi_0)/s \big)$ is supported in a ball of radius $s$ and center $s^\lambda \xi_0$. By \eqref{eq:est_sim2}, this proves \eqref{eq:prod2}. Finally, \eqref{eq:prod1} follows directly from \eqref{eq:prodH} and \eqref{eq:prodH2} with $s=t$.
\end{proof}
}

%

From \eqref{eq:error}, we know that
\[\EE |X_t|^2 \simeq t^{2\lambda(2\beta- m)},\]
hence, using this in \eqref{eq:dist}, there exists  $C>1$ so that
\begin{equation}
 e^{-c^2 C t^{2\lambda (m - 2\beta)}} \le \PP \{ |X_t| > c \} \le e^{-c^2 C^{-1} t^{2\lambda (m - 2\beta)}}.
\label{id:distribution_Xt}
\end{equation}
Thus, if $m < 2\beta$, the previous quantity tends to $1$ as $t$ goes to infinity. On the other hand, if $m = 2\beta$, the previous quantity also grows as $c$ becomes smaller. These two facts imply \ref{eq:noconv1}. If we turn now our attention to $Y_T$, we have by
the \cref{lemma:ergodic} that
\[\EE |Y_T|^2 \simeq T^{2\lambda (2\beta- \frac{1}{2} -m)  +1 }.\]
Once again,
there exists $C>1$ so that
\begin{equation}
 e^{-c^2 C T^{-2\lambda (2\beta- \frac{1}{2} -m)  -1 }} \le \PP \{ |Y_T| > c \} \le e^{-c^2 C^{-1} T^{-2\lambda (2\beta- \frac{1}{2} -m)  -1 }}.
\label{id:distribution_YT}
\end{equation}
Thus, if $m \leq 2\beta - 1/2$, the quantity above tends to $1$ as $T$ goes to infinity, which proves \ref{eq:noconv2}. This concludes the proof of \Cref{thm_main_2}.
\end{proof}

Finally, we prove \Cref{thm_main_3}, which provide the rate of convergence in probability of the limits of \Cref{thm_main_1}.

\begin{proof}[Proof of \Cref{thm_main_3}]
Start by proving \ref{in:non-averaging}. Using the notation of $P_j$ given  in \eqref{id:Pj}, we have, by \eqref{lim:Pj} and the choice of $\lambda_j$, that
\[|N^{- \lambda m_j} \mathcal N_{\beta,P_j}(f_{N,\lambda},f_{N,\lambda})    - a_j(x_0,\xi_0) | \leq O (N^{-1}) + |N^{-\lambda_j m_j} \mathcal{E}_\beta (\overline{f_{N, \lambda_j}}, f_{N, \lambda_j})|.\]
There exists $C^\prime$ such that $N \geq C^\prime/\varepsilon$,
\begin{gather*}
\PP  \{  |N^{- \lambda m_j} \mathcal N_{\beta,P_j}(f_{N,\lambda},f_{N,\lambda})    - a_j(x_0,\xi_0) | \leq \varepsilon  \} \geq \PP \{ |N^{-\lambda_j m_j} \mathcal{E}_\beta (\overline{f_{N, \lambda_j}}, f_{N, \lambda_j}) | \leq \varepsilon/2  \} \\
= 1 - \PP \{ |N^{-\lambda_j m_j} \mathcal{E}_\beta (\overline{f_{N, \lambda_j}}, f_{N, \lambda_j}) | > \varepsilon/2  \}.
\end{gather*}
By  \eqref{eq:XY} and \eqref{id:distribution_Xt}, we know that
\[  \PP \{ |N^{-\lambda_j m_j} \mathcal{E}_\beta (\overline{f_{N, \lambda_j}}, f_{N, \lambda_j}) | > \varepsilon/2  \} \le  e^{-(\varepsilon^2/ 4C) N^{2\lambda_j (m_j - 2\beta)}}. \]
From this point a simple computation allows to find $N_0$. This proves \ref{in:non-averaging}.

Next we turn our attention to \ref{in:averaging}. Using the notation of $P_j$ given  in \eqref{id:Pj}, we have, by \eqref{lim:Pj} and the choice of $\lambda_j$, that
\begin{multline*}
\left | \frac{1}{N} \int_{N}^{2N}   t^{- \lambda m_j} \mathcal N_{\beta,P_j}(f_{t,\lambda},f_{t,\lambda})  \,  \dd t  - a_j(x_0,\xi_0) \right | \\
 = O (N^{-1}) + \left | \frac{1}{N} \int_{N}^{2N}   t^{- \lambda m_j} \mathcal E_\beta (\overline{f_{t,\lambda}},f_{t,\lambda})  \,  \dd t \right |.
\end{multline*}
Arguing as in \ref{in:non-averaging}, we have that
\begin{multline*}
\PP \left\{ \left | \frac{1}{N} \int_{N}^{2N}   t^{- \lambda m_j} \mathcal N_{\beta,P_j}(f_{t,\lambda},f_{t,\lambda})  \,  \dd t  - a_j(x_0,\xi_0) \right | \leq \varepsilon
\right\} \\
\geq 1 - \PP \left\{\left | \frac{1}{N} \int_{N}^{2N}   t^{- \lambda m_j} \mathcal E_\beta (\overline{f_{t,\lambda}},f_{t,\lambda})  \,  \dd t \right | > \varepsilon / 2
\right\}.
\end{multline*}
By \eqref{eq:XY} and \eqref{id:distribution_YT}, we know that
\[
 \PP \left\{\left | \frac{1}{N} \int_{N}^{2N}   t^{- \lambda m_j} \mathcal E_\beta (\overline{f_{t,\lambda}},f_{t,\lambda})  \,  \dd t \right | > \varepsilon / 2
\right\} \le  e^{-(\varepsilon^2/ 4C) N^{-2\lambda (2\beta- \frac{1}{2} -m)  -1 }}.
 \]
Obtaining now $N_0$ is a simple computation, which proves \ref{in:averaging}. This concludes the proof of \Cref{thm_main_3}.
\end{proof}

\bibliographystyle{myplainurl}

\bibliography{references_rec_symbolo}

\end{document}